\documentclass[12pt,leqno]{amsart}
\usepackage{amsmath,amsthm}
\usepackage{amsfonts,amssymb}
\newtheorem{theorem}{Theorem}
\newtheorem{lemma}[theorem]{Lemma}

\begin{document}
\title[Convex solutions]{An elliptic pde with convex solutions}
\author{Jon Warren}
\address{Department of Statistics, University of Warwick, Coventry CV4 7AL, UK}
\email{j.warren@warwick.ac.uk}
\date{}

\maketitle

\begin{abstract}Using a mixture of classical and probabilistic techniques we investigate the convexity of solutions to the elliptic pde associated with a certain generalized Ornstein-Uhlenbeck process.
\end{abstract}

\section{Introduction and results}

We study solutions to the elliptic partial differential equation  
\begin{equation}
\label{pde}
 \frac{1}{2} \sum_{i,j= 1}^d (\delta _{ij}+x_ix_j) \frac{\partial^2 u}{\partial x_i \partial  x_j} =c, \quad \quad x \in {\mathbf  R}^d,
\end{equation}
$c$ being an arbitrary constant.  This equation arose in a probabilitic context, \cite{warren}, studying $d$ particles moving in a stochastic flow, but with each experiencing an independent Brownian perturbation, The generator of the diffusion  describing the motion of this system of particles is the operator, which we will denote by ${\mathcal A}$, appearing on the lefthand side of \eqref{pde}. The purpose of this note is to prove the convexity of certain solutions to \eqref{pde} used in \cite{warren}. 

  We will consider solutions that grow  linearly as $ |x| \rightarrow \infty$ and  admit ``boundary values''
\begin{equation}
\label{boundary}
\frac{u (x)} {|x|} \rightarrow g( x/|x|)  \text{ as } |x| \rightarrow \infty
\end{equation}
where  function $g$ defined on the sphere $S^{d-1}=\{x \in {\mathbf R}^{d}:|x|=1\} $ satisfies $ \int g(\theta) d\theta=c/\gamma_d$, where $c$ is the constant appearing on the righthandside of \eqref{pde}, and 
\begin{equation}
\gamma_d=\frac{1}{\sqrt{\pi}}
\frac{ \Gamma(((d+1)/2)}{\Gamma(d/2)}.
\end{equation} 
 We will assume that the dimension $d \geq 2$. Here the integral over the sphere is taken with respect to Lebesgue measure normalised so  $ \int 1 d\theta=1$.

Our first  result is that the ``Dirichlet problem'' is solvable for continous boundary data, with convergence to the boundary values occurring uniformly. 
\begin{theorem}
\label{dirichlet}
Suppose that  $g \in C(S^{d-1})$  and let  $c= \gamma_d \int g(\theta) d\theta$
then there exists a unique solution to the p.d.e. \eqref{pde},  with $u(0) =0$ and such that 
\[
\lim_{r\rightarrow \infty} \sup_{\theta\in S^{d-1}} | u( r\theta)/r - g(\theta)| =0.
\]
\end{theorem}

Taking the constant $c$ to be zero, this result looks at first sight as if it might be related to a Martin boundary result for the operator ${\mathcal A}$, But in fact the corresponding diffusion process is recurrent, and the only positive solutions to ${\mathcal A}u=0$ on ${\mathbf R}^d$ are  the constant solutions. Thus the Martin boundary consists of a single point at infinity, not a sphere.

It seems plausible that one could transform equation \eqref{pde} into an elliptic equation on the ball $\Omega=\{ x \in {\mathbf R}^d: |x|\leq 1\}$ with  $g$ becoming the boundary data on $\partial \Omega$, and then deduce Theorem 1 from standard results on the Dirichlet boundary problem for such equations, as described in \cite{gil-trud}. However if this were to work,  then there would have to be some solutioon corresponding to $g$ being identically constant, and no such solution to \eqref{pde} and \eqref{boundary} with $c$=0 exists.  Instead our strategy for proving Theorem 1 is to take advantage of the spherical symmetry of the opertator ${\mathcal A}$ to write a series expansion for solutions involving speherical harmonic functions. This evidently associates to any function $g$ defined on the sphere the  appropriate solution of equation \eqref{pde}. Then the more delicate part of the argument proves the uniform convergence of the solution to the boundary data making use of an  appropriate analogue of the maximum principle in the context of linear growth at infinity.

Convexity of the solutions to elliptic partial differential equations has been studied a great deal in the literature, see for example, \cite{kawohl} and \cite{korevaar}.  Here we will follow one of the established approaches to proving convexity:  making use of the fact the corresponding parabolic equation is convexity preserving. General conditions are known, see \cite{lions-musiela} and \cite{janson-tysk} that  ensure this. However in our problem we can see directly that the semigroup generated by ${\mathcal A}$ preserves convexity because the associated diffusion process can be extended to a stochastic flow of affine maps. Then to complete the argument for proving the following result we must show convergence of the solution to the parabolic equation to that of the elliptic boundary value  problem.

\begin{theorem}
Suppose that  $g \in C(S^{d-1})$   and $u \in C^2({\mathbf R}^d)$ is the solution to elliptic boundary problem \eqref{pde} and \eqref{boundary} with $u(0)=0$.  Then $u$ is convex if and only if   $v \in C({\mathbf R}^d) $  given by
\[
v(x)= |x|g(x/|x|)   \quad x \in {\mathbf R}^d,
\]
is convex also. 
\end{theorem}

\section{Separation of variables and properties of the radial equation}

We may rewrite the operator ${\mathcal A}$ in spherical co-ordinates as
\begin{multline}
\label{sphere1}
{\mathcal A}= \frac{r^2}{2}\frac{\partial^2}{\partial r^2}  +\frac{1}{2}
\nabla^2 = \\
\frac{1}{2}(1+r^2 )\frac{\partial^2}{\partial r^2}+ \frac{(d-1)}{2r}\frac{\partial}{\partial r} +\frac{1}{2 r^{2}}\Delta_{S^{d-1}}= {\mathcal A}_{R}+\frac{1}{2 r^{2}}\Delta_{S^{d-1}},
\end{multline}
where $\Delta_{S^{d-1}}$ is the  Laplace-Beltrami  operator on the sphere $S^{d-1}$. The evident spherical symmetry suggests a solution by the separation of variables, taking the form 
\begin{equation}
\label{series}
u(x)=u(r\theta)= \sum_{l\geq 0} f_l ( r ) g_l(\theta).
\end{equation}
Suppose that  $g \in L^2(S^{d-1})$ and  take $g_l$ to be  the   projection in $L^2(S^{d-1})$ of $g$ onto the space of spherical harmonic functions of degree $l$, see \cite{m}. Then $g_l$ satisfies 
\begin{equation}
\Delta_{S^{d-1}} g_l = -l(l+d-2) g_l,
\end{equation}
and consequently for $l \geq 1$, we would like     $f_l$ to  solve
\begin{equation}
\label{radialDE}
{\mathcal A}_{R}f_l-\frac{l(l+d-2)}{2r^2}f_l=0
\end{equation}
with $f_l(r)/r  \rightarrow 1$  as $r \rightarrow  \infty$ and $f_l(0+)=0$.  In fact such $f_l$ may be expressed in
terms of hypergeometric functions, see Lemma \ref{sol}.

For $l=0$  we define $f_l$ differently, one reason for this being that non-constant solutions to \eqref{radialDE} with $l=0$ all have a singularity at the origin.  Instead we take $f_0$ to  solve
\begin{equation}
\label{radialDE0}
{\mathcal A}_{R}f_0=\gamma_d
\end{equation}
with $f_0(r)/r  \rightarrow 1$ as  $r \rightarrow  \infty$ and $f_0(0+)=0$. 
This has a solution
\begin{equation}
\label{exact0}
f_0( r )= 2\gamma_d \int_0^r \left(\frac{u^2}{1+u^2}\right)^{-(d-1)/2} \int_0^u
\frac{v^{d-1}}{(1+v^2)^{(d+1)/2}} dv du
\end{equation}
which may  be verified  by simple calculus, noting that
\[
\int_0^\infty
\frac{v^{d-1}}{(1+v^2)^{(d+1)/2}} dv = \frac{1}{2\gamma_d}.
\] 

Using Euler's integral representation of the hypergeometric function
it is straightforward to check, see Lemma \ref{sol}, that $f_l( r )$  decays  to $0$
geometrically fast  for $r$ in compact sets as $l$ tends to infinity. On the otherhand, $g_l(\theta) $ grows at most polynomially as $ l$ tends to
infinity, as can be seen from the integral representation for $g_l$ ( page 42, \cite{m}). In conjunction these facts  guarantee that the series \eqref{series} converges uniformly 
 on compact sets of ${\mathbf R}^d$  and does indeed define a smooth
 solution to  ${\mathcal A} u =c$ except possibly at the origin.  But since $\{0\}$ is a polar set for the diffusion associated with ${\mathcal A}$,  any bounded solution to ${\mathcal A} u =c$ in the  punctured ball $\{ x \in {\mathbf R}^d: 0 <|x|<1\}$ extends to a solution on the entire ball, and so  \eqref{series} defines a solution on all of ${\mathbf R}^d$.

\begin{lemma}
\label{sol}
The solution to 
\[
{\mathcal A}_{R} f -\tfrac{l(l+d-2)}{2r^2}f=0,
\]
satisfying boundary conditions $f(0)=0$ and $f(r)/r \rightarrow 1$ as $ r \rightarrow \infty$ is 
\[
f( r )= f_l( r)= r^l \frac{\Gamma( (l+d+1)/2 )\Gamma(l/2)}{\Gamma(l+d/2)\Gamma(1/2)}\;_2F_1(l/2,(l-1)/2;l+d/2;-r^2)
\]
Moreover for each $R>0$, there exists $\delta_R \in (0,1)$ so that
\[
\sup_{r \leq R} f_l( r ) \leq \delta_R^l \text{ for all sufficiently large l}.
\]  
\end{lemma}
\begin{proof} Substituting $f( r )= r^l y(-r^2)$ and $x=-r^2$ into 
\[
\tfrac{1}{2} (1+r^2) f^{\prime\prime} +\tfrac{d-1}{2} f^\prime -\tfrac{l(l+d-2)}{2r^2}f=0
\]
gives 
\[
x(1-x) y^{\prime\prime} + \{l+\tfrac{d}{2}-x(l+\tfrac{1}{2} \} y^\prime- \tfrac{l(l-1)}{4} y=0,
\]
which is the standard form of the hypergeometric equation with parameters $a=l/2$, $b=(l-1)/2$ and $c=l+d/2$. The boundary condition $f(0)=0$ is satisfied by taking $y(x)$ proportional to $_2F_1(a,b;c;x)$. Now  to choose the constant of proportionality to get the behaviour as $r \rightarrow \infty$ correct we combine Pfaff's transformation with Gauss's formula for $_2F_1(a,b;c;1)$ to deduce that
\[
\lim_{x \rightarrow - \infty} (1-x)^b \;_2F_1(a,b;c;x) = \;_2F_1(c-a,b;c;1)= \frac{\Gamma( c )\Gamma(a-b)}{\Gamma(c-b)\Gamma(a)}.
\]
Next using Euler's integral representation for the hypergeometric function
\[
f_l( r )= r^l \frac{\Gamma(l/2)}{\Gamma((l-1)/2)\Gamma(1/2)} \int_0^1 t^{(l-3)/2}(1-t)^{(d+l-1)/2} (1+r^2t)^{-l/2}dt.
\]
Now the ratio of gamma functions appearing here grows sublinearly with $l$, whereas we can estimate the integral as being less than
\[
\sup_{0\leq t \leq 1} \left(\frac{ 1-t}{1+r^2t} \right)^{l/2} \leq \left( \frac{1}{1+r^2}\right)^{l/2}.
\]
Consequently the statement of the lemma holds choosing $\delta_R>R/\sqrt{1+R^2}$.
\end{proof}

\section{The associated diffusion}

Associated with the operator ${\mathcal A}$ is a diffusion and we will make use of this to study solutions of \eqref{pde}. In fact the SDE corresponding to $ {\mathcal A}$ is linear, and consequently the  diffusion can be constructed explicitly as in the following lemma.
 Of particular importance is that this representation of the diffusion actually defines a stochastic flow of affine maps of ${\mathbf R}^d$.

\begin{lemma}
\label{yflow}
Let $B$ be a standard one dimensional Brownian motion, and $W$ a
standard Browninan motion in ${\mathbf R}^{d}$. For 
$x\in {\mathbf R}^d$, let 
\begin{equation}
\label{defX}
X^x(t)= x\exp\{B(t)-t/2\} + \int_0^t \exp\{(B(t)-B(s))-(t-s)/2\}dW(s)
\end{equation}
then $\bigl( X^x(t); t \geq 0)$ is a diffusion with generator ${\mathcal  A}$ starting from $x$. 
\end{lemma}

\begin{proof}
This follows by applying  It\^{o}'s formula to $X^x$.
\end{proof}
It is easy to see from this lemma that the diffusion is recurrent rather than transient.  Indeed we have for every $x \in {\mathbf R}^d$, as $t\rightarrow \infty$,
\begin{multline}
\label{coninv}
X^x(t)=x\exp\{B(t)-t/2\} + \int_0^t \exp\{(B(t)-B(s))-(t-s)/2\}dW(s) \stackrel{\text{law}}{=} \\
x\exp\{B(t)-t/2\} + \int_0^t \exp\{B(s)-s/2\}dW(s)
 \stackrel{\text{a.s.}}{\rightarrow} \int_0^\infty \exp\{B(s) -s/2\}dW(s),
\end{multline}
where the  last stochastic integral is almost surely convergent because its quadratic variation is almost surely finite. It is the fact the associated diffusion is not transient that makes the treatment of the Dirichlet problem for ${\mathcal A}$  somewhat non-standard.

The process $X^x$  defined by \eqref{defX} is an example of a generalized Ornstein-Uhlenbeck process. See \cite{cpy} for  general discussion of these processes and in particular their invariant measures. The particular case of the generalized OU process constructed from two one-dimensional Brownian motions, which corresponds to \eqref{defX} with $d=1$,   was studied in \cite{yor}. There is a close relationship between the  generalized OU processes and exponential functionals of L\'{e}vy  processes,  in our case,  exponential functionals of Brownian motion. These  have been are extensively studied, see the survey article, \cite{my}.
In particular we will have need of the folowing observations. The invariant measure given at \eqref{coninv} can be re-written in the form
\begin{equation}
\int_0^\infty \exp\{B(s) -s/2\}dW(s)  \stackrel{\text{law}}{=} W \bigl( A_\infty \bigr) \stackrel{\text{law}}{=}
\sqrt{ A_\infty} W(1),
\end{equation}
where $A_\infty$ denotes the exponential functional $\int_0^\infty \exp\{2B(s) -s\} ds$. The distribution of this latter random variable is known to be a stable distribution of index $1/2$, see \cite{dufresne},  also  Theorem 6.2 of \cite{my},  and consequently ${\mathbf E} [ A_\infty^{p/2}]$ is finite for $0<p<1$ and infinite if $p=1$. It follows that if $X(\infty)$ is a ${\mathbf R}^d$ valued  random variable whose distribution is the invariant measure at \eqref{coninv}, then,
\begin{equation}
{\mathbf E} [ |X(\infty)|^{p}]< \infty \text{ for $p<1$, and } {\mathbf E} [ |X(\infty)|]=\infty.
\end{equation}
Moreover, the convergence at \eqref{coninv} occurs in $L^p$ for every $p<1$.
On the otherhand for every finite time $t<\infty$ we have
\begin{equation}
\label{sqfinite}
 {\mathbf E} [ |X^x(t)|^2]<\infty.
\end{equation}

\section{Proof of Theorem 1}

In order to prove Theorem  \ref{dirichlet} we must  show that the solution $u$, given by the series \eqref{series}, has the  correct boundary behaviour. If $g$ is a finite linear combination of spherical harmonic functions then this follows immediately from the asymptotic behaviour of $f_l$. However in general it is more difficult to verify the limit behaviour of $u$. The key tool we use is the following result which plays the role of a maximum principle in our setting.

\begin{lemma}
\label{max}
There exists a constant $K$ such that for every  $g\in C(S^{d-1}) $ satisfying $ \int_{S^{d-1}} g d\theta=0$   the  function  $u$ given by \eqref{series} and corresponding to $g$ satisfies 
\[
|u(x)| \leq K(1+|x|) \sup_{\theta \in S^{d-1}} |g(\theta)|  \qquad \text{ for all } x \in {\mathbf R}^d. 
\]
\end{lemma}
 
Admitting this result we can prove the covergence statement of  Theorem \ref{dirichlet} as follows. Fix an arbitrary $g \in C(S^{d-1})$.  Finite linear combinations of spherical harmonics are dense in $C(S^{d-1})$ by the Stone-Weierstrass Theorem, and hence given any $ \epsilon>0$ we can find $g_\epsilon$, a finite linear combination of spherical harmonics, satisfying  $ \int_{S^{d-1}} g_\epsilon d\theta= \int_{S^{d-1}} g d\theta$ and with
\[
|| g_\epsilon -g ||_\infty \leq \epsilon.
\]
But then if $u_\epsilon$ is the  solution to \eqref{pde} which corresponds to $g_epsilon$ given by a finite series of the form \eqref{series},  as we have remarked already,
\[
\lim_{r\rightarrow \infty } \sup_{\theta \in S^{d-1}} | u_\epsilon( r\theta)/r - g_\epsilon (\theta)| =0.
\]
Now $u-u_\epsilon$ corresponds to $g-g_\epsilon$, which has mean $0$, and applying the previous lemma to this  we obtain
\[
|u(x)- u_\epsilon(x)| \leq K\epsilon( 1+ |x|),
\]
and hence  
\[
\limsup_{r\rightarrow \infty} \sup_{\theta\in S^{d-1}} | u( r \theta)/r - g(\theta)| \leq (K+1)\epsilon.
\]
Since $\epsilon$ is arbitrary this proves the desired uniform convergence.

\begin{proof}[Proof of Lemma \ref{max}]
We begin by  solving the equation ${\mathcal A}_{R} h =0 $.  By elementary means  we find that the general solution is a linear combination of a constant and the function
\begin{equation}
h( r )= \int_1^r \left(\frac{1+u^2}{u^2} \right) ^{(d-1)/2} du.
\end{equation}
Notice that $h(r)/r \rightarrow 1$ as $ r \rightarrow \infty$. Now, for $R>|x|>r$,  let
\[
\tau_{r,R}= \inf\{ t>0: X_t^x\not\in (r,R)\}.
\]
Taking expectations of the martingale $h(|X^x_{t \wedge \tau_{r,R}}|)$, we obtain,
\begin{equation}
\label{exitprob}
{\mathbf P}( |X^x_{\tau_{r,R}}|=R)= \frac{h(|x|)-h( r )}{h( R )-h( r )}.
\end{equation}

Now note that  for each $x$, $u(x)$ varies continuously with $g\in C(S^{d-1})$. In fact there exist constants $K_R$ so that
\begin{equation}
\label{cts}
\sup_{|x| \leq R} |u(x)| \leq K_R \sup_{\theta \in S^{d-1}} |g(\theta)|
\end{equation}
as can be seen by estimating the terms in the series \eqref{series} using Lemma \ref{sol}.
Consequently  it is enough to prove the  inequality for $g$ belonging to the dense subset consisting of  $g\in C(S^{d-1})$ formed of finite linear combinations of spherical harmonics with $\int_{S^{d-1}} g d\theta=0$. Fix such a $g$ and let $u$ be the corresponding solution  of ${\mathcal A} u =0$. Considering the martingale $u( X^x_{t \wedge \tau_{1,R}})$, where $1<|x|<R$, we obtain 
\[
u(x)= {\mathbf E} [ u( X^x_{ \tau_{1,R}})],
\]
whence, using \eqref{exitprob},
\begin{equation}
\label{estimate}
|u(x)|  \leq \sup_{|y|=1} |u(y)|+  \frac{h(|x|)}{h( R )}\sup_{|y|=R} |u(y)|.
\end{equation}
Recall that as we have observed previously since $u$ is formed from a  finite linear combination  of spherical harmonics,
\[
\lim_{r\rightarrow \infty} \sup_{\theta\in S^{d-1}} | u( r\theta)/r - g(\theta)| =0.
\]
  Consequently, letting $R \rightarrow \infty$ in \eqref{estimate} we obtain,
\[
|u(x)|  \leq \sup_{|y|=1} |u(y)|+  h(|x|)\sup_{\theta\in S^{d-1}} | g(\theta)|.
\]
Now we apply the estimate  \eqref{cts} to the first of these terms, and  we deduce the statement of the  lemma holds if $K$ is chosen greater than both $\sup_{r \geq 1} h( r )/r$ and $K_1$.

\end{proof}

It remains to prove the uniqueness assertion of the theorem. This we can do adapting the argument just used in the proof of the lemma.
Suppose that $u_1$ and $u_2$ are two solutions to ${\mathcal A} u=0$  satisfying
\[
\lim_{r\rightarrow \infty} \sup_{\theta\in S^{d-1}} | u_i( r\theta)/r - g(\theta)| =0
\]
for the same choice of $g$.
Then $u= u_1-u_2$ solves ${\mathcal A} u=0$  with 
\begin{equation}
\label{zero}
\lim_{r\rightarrow \infty} \sup_{\theta\in S^{d-1}} | u( r\theta)/r | =0.
\end{equation}
Considering the martingale $u( X^x_{t \wedge \tau_{r,R}})$ we obtain 
\[
u(x)= {\mathbf E} [ u( X^x_{ \tau_{r,R}})],
\]
whence, using \eqref{exitprob},
\begin{equation}
\label{estimate}
|u(x)|  \leq \sup_{|y|=r} |u(y)|+  \frac{h(|x|)-h( r )}{h( R )-h( r )}\sup_{|y|=R} |u(y)|.
\end{equation}
Now letting $R\rightarrow \infty$, holding $r$ fixed, and using \eqref{zero}, gives
\[
|u(x)|  \leq \sup_{|y|=r} |u(y)|
\]
But then letting $ r \downarrow 0$ and noting $u(0)=0$ we deduce $u$ is identically zero.

\section{Proof of Theorem 2}

We now define the semigroup $(P_t; t \geq 0)$ via 
$P_t u (x)= {\mathbf E}[ u(X^x(t))]$ 
whenever 
$u$ is such that the random variable $u(X^x(t))$ is integrable for all $x\in {\mathbf R}^d$.   Recall, in particular, that ${\mathbf E}[|X^x(t)|^2] < \infty$.

Each random map $x \mapsto X^x(t)$ is affine and consequently if $u$ is  a convex function then the random function  $x \mapsto u(X^x(t))$ is convex with probability one also.  Taking expectations we
have, for any $x,y \in {\mathbf R}^d$ and $\alpha\in [0,1]$,
\begin{multline*}
P_tu (\alpha x+(1-\alpha)y)= {\mathbf E}[ u( \alpha X^x(t)+(1-\alpha)X^y(t))]  \\
\leq  {\mathbf E}[ \alpha u(X^x(t))+(1-\alpha) u(X^y(t))]= \alpha P_t u (x)+ (1-\alpha) P_tu (y),
\end{multline*}
and thus $P_t$  preserves convexity.
This will be a key ingredient in the proof of our second theorem. We note in passing that the semigroup of any generalized OU process is convexity preserving.

 Our strategy for  the proof of  Theorem 2 is to study the behaviour of  $P_t v$ as $t \rightarrow \infty$ where $v(x)= |x|g(x/|x|) $. To begin, first note that the  probabilistic analogue  of  \eqref{sphere1}
is the skew-product decomposition for  the diffusion $(X^x(t);t \geq 0 )$:
\begin{equation}
\label{sphere2}
X^x(t)= R^{(r)}(t) \Theta \left( \int_0^t \frac{ds}{R^{(r)}(s)^2 ds} \right)
\end{equation}
where   $R^{(r)}(t)= |X^x(t)|$ is a diffusion on $(0,\infty)$  with generator ${\mathcal A}_R$ starting from $r=|x|\neq 0$, and $(\Theta(t); t \geq 0)$ a Brownian motion on the sphere $S^{d-1}$ starting from $x/|x|$. An elegant argument for establishing this skew-product  is to write  $X^x(t)$ as a time change
\begin{equation}
X^x(t)= e^{B(t)-t/2} \hat{W} \left( \int_0^t e^{-2B(s)+s}ds\right)
\end{equation}
of a $d$-dimensional Brownian motion $\hat{W}$ satisfying $\hat{W}(0)=x$, and then apply the usual 
skew-product decomposition of $d$-dimensional Brownian motion to  $\hat{W}$.

Equations \eqref{radialDE} and \eqref{radialDE0} imply that  the processes
\begin{equation}
f_l\bigl(R^{(r)}(t)\bigr) \exp\left( - \frac{l(l+d-2)}{2} \int_0^t \frac{ds}{R^{(r)}(s)^2} \right)
\end{equation}
for $l\geq 1$,  and,
\begin{equation}
 f_0\bigl(R^{(r)}(t)\bigr) -\gamma_d t
\end{equation}
are local martingales. In fact they are true martingales because $f_l^\prime$ being bounded together with \eqref{sqfinite} implies their quadratic variations are square integrable.

 Now  define    $f_l(t,r)$  by,
\begin{equation}
\label{time-evol}
f_l(t,r)= {\mathbf E}\left[  R^{(r)}(t) \exp\left( - \frac{l(l+d-2)}{2} \int_0^t \frac{ds}{R^{(r)}(s)^2} \right) \right].
\end{equation}

\begin{lemma}
\label{lemma1con}
For $l\geq 1$ we have for all $ r\geq0$,
\[
\lim_{t\rightarrow \infty} f_l(t,r)= f_l(r).
\]
Moreover we have $ f_l(r) \leq f_l(t,r)\leq r $ for all $t\geq 0$ and $l \geq 1$.
The case $l=0$ satisfies
\[
\lim_{t\rightarrow \infty} \bigl(f_0(t,r)-\gamma_d t \bigr)= f_0(r)+\lambda_d,
\]
for all $ r\geq0$, where $\lambda_d$ is a constant not depending on $r$.
\end{lemma}

\begin{proof}
 Fix $l \geq 1$.  Since $f_l(r)/r \rightarrow 1$ as $ r \rightarrow \infty$, for any $\epsilon>0$ there exists a $K$ so that for all $r \geq 0$,
\[
(1-\epsilon)f_l(r)-K \leq r \leq (1+\epsilon) f_l(r)+K.
\]
Replacing $r$ by $R^{(r)}(t)$, multiplying by $\exp\left( - \frac{l(l+d-2)}{2} \int_0^t \frac{ds}{R^{(r)}(s)^2} \right)$ and taking expectations, we deduce that
\begin{equation}
\label{sand}
(1-\epsilon)f_l(r)-K \delta_l(t,r) \leq f_l(t,r) \leq (1+\epsilon) f_l(r)+K  \delta_l(t,r) ,
\end{equation}
where  $\delta_l(t,r) ={\mathbf E}\left[ \exp\left( - \frac{l(l+d-2)}{2} \int_0^t \frac{ds}{R^{(r)}(s)^2} \right) \right]$. Now the diffusion $X^x(t)$ being recurrent implies that $\int_0^\infty \frac{ds}{R^{(r)}(s)^2} =\infty$ with probability one, and hence $\delta_l(t,r) \rightarrow 0$ as $t \rightarrow \infty$. Thus, in \eqref{sand}, if we let $t \rightarrow \infty$ and then $\epsilon \downarrow 0$, we deduce that $\lim_{t\rightarrow \infty} f_l(t,r)= f_l(r)$ as desired.

For $l\geq 1$ applying It\^{o}'s formula to
\[
R^{(r)}(t) \exp\left( - \frac{l(l+d-2)}{2} \int_0^t \frac{ds}{R^{(r)}(s)^2} \right)
\]
shows this process to a supermartingale, and hence $f_l(t,r)$ is a decreasing function of $t$. This shows that
$ f_l(r) \leq f_l(t,r)\leq   f_l(0,r)=r $.

Set $\hat{f}_0(r)=r-f_0(r)$.  Using \eqref{exact0}, it is easy to check that there exists constants $A$ and $B$ so that
\begin{equation}
\label{logbound}
|\hat{f}_0(r)| \leq A+B \log(1+r).
\end{equation}
Now 
\begin{multline}
\label{loglimit}
{\mathbf E}\bigl[ \hat{f}_0(R^{(r)}(t))\bigr]={\mathbf E}\bigl[ \hat{f}_0(|X^{x}(t)|)\bigr] 
\\
={\mathbf E}\left[ \hat{f}_0\left(\left|x\exp\{B(t)-t/2\} + \int_0^t \exp\{(B(t)-B(s))-(t-s)/2\}dW(s) \right|\right)\right]  \\
={\mathbf E}\left[ \hat{f}_0\left(\left|x\exp\{B(t)-t/2\} + \int_0^t \exp\{B(s)-s/2\}dW(s) \right|\right)\right]  \\
\rightarrow {\mathbf E}\left[ \hat{f}_0\left(\left|\int_0^\infty \exp\{B(s)-s/2\}dW(s) \right|\right)\right].
\end{multline}
This convergence of expectations is justified by the uniform integrability of the random variables which follows from  the bound \eqref{logbound} and the fact that the fact that the convergence at \eqref{coninv} occurs in $L^p$ for any $0<p<1$.  Now define the constant $\lambda_d$ to be the value of the limit at \eqref{loglimit}, which doesnt depend on $r$. Then we have
\begin{multline}
f_0(t,r)= {\mathbf E}\bigl[ R^{(r)}(t)\bigr]= {\mathbf E}\bigl[ f_0(R^{(r)}(t))+ \hat{f}_0(R^{(r)}(t))\bigr] \\
=f_0(r)+\gamma_dt +{\mathbf E}\bigl[ \hat{f}_0(R^{(r)}(t))\bigr] \rightarrow f_0(r)+\gamma_dt+\lambda_d.
\end{multline} 

\end{proof}

In the following lemma we establish the convergence of  (a shift of) $P_tv$ to the solution $u$ of the elliptic equation. We expect that this convergence to be  locally uniform, but its enough for our purposes to prove it in a weaker $L^2$ sense. 

\begin{lemma}
\label{l2con}
Suppose that $ g \in C(S^{d-1})$ and let   $c=\gamma_d \int g(\theta) d\theta$, and  $b=\lambda_d \int g(\theta) d\theta.$  Let $v(x)= |x|g(x/|x|) $ for $ x \in {\mathbf R}^d$ and 
let $u$ be the solution of \eqref{pde} corresponding to $g$. 
Then, as $t\rightarrow \infty$,
\[
\int_{S^{d-1}}(P_t v (r\theta)  - u(r\theta) -ct-b )^2 d \theta \rightarrow 0,
\]
for every $r>0$.
\end{lemma}

\begin{proof}
Letting $g_l$ be the projection of $g$ into the subspace of spherical harmonics of degree $l$ as usual, we claim we can expand   
$P_tv$  as a series,
\begin{equation}
\label{time-series}
P_tv (r\theta) = \sum_{l=0}^\infty f_l(t, r) g_l(\theta), 
\end{equation}
with the series converging in   $L^2(S^{d-1}( r ))$ for each $r>0$. This convergence is guaranteed by the inequality  $0\leq f_l(t, r) \leq r$.

To verify the claim
that \eqref{time-series} is valid, first note it holds for $g$  that are a finite linear combination of spherical harmonics, by  virtue of the skew product \eqref{sphere2}, the fact that $g_l$ is an eigenfunction of the Laplacian on the sphere, and the definition \eqref{time-evol} of $f_l(t,r)$. Now consider, for a fixed $r>0$ and $t>0$, the applications,
\[
g \in C(S^{d-1}) \mapsto  P_tv (r\cdot) \in L^2(S^{d-1}),
\]
and 
 \[
g \in C(S^{d-1}) \mapsto  \sum_{l=0}^\infty f_l(t,r) g_l(\cdot)  \in L^2(S^{d-1}).
\]
Both are continous (equipping  $C(S^{d-1})$ with the uniform norm) and they agree on the dense subspace of finite linear combinations of spherical harmonics. Thus \eqref{time-series} holds for any $g\in C(S^{d-1})$.

With the help of \eqref{time-series} we can now compute, noting $g_0= \int_{S^{d-1}} g(\theta)d\theta$,
\begin{multline*}
\int_{S^{d-1}}(P_t v (r\theta)  - u(r\theta) -ct-b )^2  \; d\theta
\\
= ( f_0(t,r)g_0-f_0(r)g_0-ct-b)^2+ \sum_{l=1}^\infty ( f_l(t,r)-f_l(r))^2 ||g_l||^2_{S^{d-1}},
\end{multline*}
which tends to $0$ as $t \rightarrow \infty$ using Lemma \ref{lemma1con} and the Dominated Convergence Theorem.
\end{proof}

\begin{proof}[Proof of Theorem 2]
Recall that $v$ being convex implies that $P_t v $ is convex also for every $t \geq 0$. Because $L^2$ convergence implies almost everywhere convergence along some subsequence, it follows from Lemma \ref{l2con} that ,for all but  a null set of $x,y \in {\mathbf R}^d$ and $ \alpha \in [0,1]$, 
\[
u(\alpha x+ (1-\alpha)y) \leq  \alpha u(x)+ (1-\alpha)u(y).
\]
But $u$ is continuous so this inequality extends to all  $x,y \in {\mathbf R}^d$ and $ \alpha \in [0,1]$.

To prove  the converse implication, consider arbitrary $x,y \in {\mathbf R}^d\setminus\{0\}$ and $\alpha \in [0,1]$ with $ \alpha x +(1-\alpha)y \neq 0$. Then $u$ being conxex implies that, for every $r>0$,
\[
\alpha u(rx)+ (1-\alpha) u(ry) \geq u( \alpha rx+(1-\alpha)ry).
\]
Dividing through by $r$, and then letting $ r \rightarrow \infty$, we obtain from \eqref{boundary} that
\[
\alpha |x| g(x/|x|)+ (1-\alpha) |y| g(y/|y|) \geq |\alpha x+(1-\alpha)y| g\left( \frac{ \alpha x+(1-\alpha)y}{| \alpha x+(1-\alpha)y|}\right)
\]
which in view of the definition of $v$ implies that it is convex.

\end{proof}

\end{document}